\documentclass[a4paper, 11pt, leqno]{amsart}
\usepackage{amsmath,amsthm,amsfonts,amssymb,graphicx,mathrsfs,esint,longtable,mathtools, todonotes,enumitem,verbatim}
\usepackage[colorlinks=true, linkcolor=red, urlcolor=red, citecolor=green]{hyperref}
\usepackage{cleveref}
\usepackage{color}

\newcommand{\h}{{\mathcal{H}}}

\newcommand{\R}{{\mathbb R}}

\def\XXint#1#2#3{{\setbox0=\hbox{$#1{#2#3}{\int}$ }
\vcenter{\hbox{$#2#3$ }}\kern-.6\wd0}}

\usepackage[shortcuts]{extdash}
\usepackage{physics}
\allowdisplaybreaks

\newcommand{\N}{\mathbb{N}}

\newcommand{\e}{\varepsilon}

\DeclareMathOperator{\diam}{diam}

\newcommand{\Ha}{\mathcal{H}}
\DeclareMathOperator{\Mod}{mod}
\newcommand{\modd}{\textnormal{mod}}

\newcommand{\dist}{\textnormal{dist}}

\newtheorem{theorem}{\textbf{THEOREM}}[section]
\newtheorem{lemma}[theorem]{\textsc{Lemma}}
\newtheorem{proposition}[theorem]{\textsc{Proposition}}
\newtheorem{corollary}[theorem]{\textsc{Corollary}}
\theoremstyle{remark}

\theoremstyle{definition} 
\newtheorem{definition}[theorem]{\textsc{Definition}}
\newtheorem{question}[theorem]{\textsc{Question}}
{\theoremstyle{remark} }

\usepackage{xpatch}
\makeatletter
\AtBeginDocument{\xpatchcmd{\@thm}{\thm@headpunct{.}}{\thm@headpunct{}}{}{}}
\makeatother

\def\charfn_#1{{\raise1.2pt\hbox{$\chi_{\kern-1pt\lower3pt\hbox{{$\scriptstyle#1$}}}$}}}
\def\leq{\leqslant }
\def\geq{\geqslant }

\def\XXint#1#2#3{{\setbox0=\hbox{$#1{#2#3}{\int}$}
\vcenter{\hbox{$#2#3$}}\kern-.5\wd0}}

\def\le {\leqslant}
\def\ge {\geqslant}

\begin{document}

\title[Infinitesimally metric measures]{Uniformization with infinitesimally metric measures} 
\author{Kai Rajala, Martti Rasimus, and Matthew Romney} 
\let\thefootnote\relax\footnote{\emph{Mathematics Subject Classification 2010:} Primary 30L10, Secondary 30C65, 28A75, 51F99.}
\thanks{All authors were supported by the Academy of Finland, project number 308659. }
\begin{abstract}
We consider extensions of quasiconformal maps and the uniformization theorem to the setting of metric spaces $X$ homeomorphic to $\R^2$. Given a measure $\mu$ on such a space, we introduce \emph{$\mu$-quasiconformal maps} $f:X \to \R^2$, whose definition involves deforming lengths of curves by $\mu$. We show that if $\mu$ is an \emph{infinitesimally metric measure}, i.e., it satisfies an infinitesimal version of the metric doubling measure condition of David and Semmes, then such a $\mu$-quasiconformal map exists. We apply this result to give a characterization of the metric spaces admitting an \emph{infinitesimally quasisymmetric} parametrization.  
\end{abstract}

\maketitle

\renewcommand{\baselinestretch}{1.2}


\section{Introduction} \label{sec:introduction}

The \emph{quasisymmetric uniformization problem} asks one to characterize, as meaningfully as possible, those metric spaces which may be mapped onto a domain in the Euclidean plane, or the $2$-sphere, by a quasisymmetric homeomorphism. Informally, a mapping is \emph{quasisymmetric} if it roughly preserves the relative distance between triples of points. See \Cref{sec:proof_iqs} for the precise definition.

Significant results on the uniformization problem, such as the Bonk--Kleiner theorem \cite{BK:02} and its extensions in \cite{Wil:08} and \cite{Wil:10}, have been obtained for surfaces that are non-fractal, i.e., their $2$-dimensional Hausdorff measure is locally finite. These spaces carry enough rectifiable paths for classical methods such as conformal modulus to be applicable. By \emph{surface}, we mean a 2-manifold equipped with a continuous metric.

In contrast, the class of fractal surfaces is too general for the standard methods. Consequently, understanding the quasisymmetric uniformization of fractal surfaces has proved extremely difficult. Any progress is desirable, especially due to applications to geometric group theory (cf. \cite{Bon:06}, \cite{Kle:06}) and complex dynamics  (cf. \cite{BonMey:17}). 

The usual method for constructing quasisymmetric maps is to first show the existence of some \emph{conformal} or \emph{quasiconformal} map in the spirit of the classical uniformization theorem. Then, if the underlying surface has good geometric properties, one can use quasiconformal invariants to show that such a map is actually quasisymmetric.  

A fundamental difficulty in extending this method to fractal surfaces is the lack of a suitable definition of quasiconformality. The classical metric definition (see \Cref{sec:proof_iqs}) is too weak to lead to a satisfactory theory in this generality. The geometric definition (see \Cref{sec:notation}) requires the existence of many rectifiable paths, which need not be the case for fractal surfaces.

In \Cref{sec:notation} we propose the definition of \emph{$\mu$-quasiconformality} for homeomorphisms $f\colon X \to \R^2$, depending on a measure $\mu$ on $X$. This is a modification of the geometric definition: we deform the metric on $X$ using $\mu$ to obtain the \emph{$\mu$-length} of a curve, and we define the corresponding \emph{$\mu$-modulus} of a family of curves in $X$. A homeomorphism $f$ is $\mu$-quasiconformal if the $\mu$-modulus of every family of curves in $X$ is comparable to the conformal modulus of its image under $f$ in $\R^2$. 


A quasisymmetric map $f\colon X \to \R^2$ is $\mu$-quasiconformal when $\mu$ is the pullback of the Lebesgue measure on $\R^2$. Our goal is to find measures $\mu$ on a given space $X$ for which the existence of $\mu$-quasiconformal maps can be shown. 

In \Cref{sec:proof_mu_quasiconformal} we introduce the notion of \emph{infinitesimally metric measure} on $X$. These correspond to the metric doubling measures of David and Semmes \cite{DS:90}, \cite{LohRajRas:18}, the correspondence being similar to the one between metrically quasiconformal (MQC) maps and quasisymmetric (QS) maps, where the former is an infinitesimal condition and the latter is a global condition. Metric doubling measures can be used to produce quasisymmetric maps via deformation of the metric on $X$. Our first main result shows that a $\mu$-quasiconformal map exists if $\mu$ is an infinitesimally metric measure. 

\begin{theorem} \label{thm:mu_quasiconformal}
Let $X$ be a metric space homeomorphic to $\mathbb{R}^2$ which supports an infinitesimally metric measure $\mu$. Then there exists a $\mu$\-/quasiconformal map $f \colon X \to \Omega$, 
where $\Omega=\mathbb{D}\subset \mathbb{R}^2$ or $\Omega=\mathbb{R}^2$. 
\end{theorem}

To prove \Cref{thm:mu_quasiconformal}, we first show that the metric $d$ on $X$ can be deformed using $\mu$ to yield a ``quasiconformally equivalent'' metric $q$ that has locally finite Hausdorff $2$-measure. Then, we apply the uniformization theorem in \cite{Raj:16} to obtain a quasiconformal map $(X,q) \to \R^2$. Composing, we then get the desired $\mu$-quasiconformal map.    

In view of the correspondence between infinitesimally metric measures and metric doubling measures, it is natural to attempt to characterize the class of metric spaces $X$ that admit metrically quasiconformal maps $f\colon X \to \R^2$ in terms of infinitesimally metric measures. However, it turns out that the existence of such maps can be rather arbitrary unless strong conditions are imposed on $X$.

Instead, we consider the notion of \emph{infinitesimally quasisymmetric}  (I-QS) mapping (\Cref{def:qs}). Such maps form an intermediate class between those of MQC and QS maps. In our second main result, we characterize the metric spaces which admit such maps into $\R^2$ as the spaces that carry infinitesimally metric measures with suitable properties. 

\begin{theorem} \label{thm:iqs}
Let $X$ be a metric space homeomorphic to $\mathbb{R}^2$. There exists an infinitesimally quasisymmetric map $f\colon X \to \Omega$, where $\Omega=\mathbb{D}$ or $\Omega=\R^2$, if and only if $X$ is infinitesimally linearly locally connected and supports an infinitesimally metric measure $\mu$ such that $(X,\mu)$ is infinitesimally Loewner. 
\end{theorem}

See \Cref{sec:proof_iqs} for definitions. The proof combines \Cref{thm:mu_quasiconformal} with estimates for the $\mu$-modulus that generalize the modulus estimates in \cite{HK:98}. 

One motivation for our work is to understand the conformal geometry of metric surfaces in the absence of strong geometric assumptions such as Ahlfors regularity, linear local connectedness and the Loewner condition (see \Cref{sec:proof_iqs}). In \Cref{sec:exm}, we present four examples to illustrate possible behaviors of metric surfaces under weaker geometric assumptions. We remark that, while the main theorems of this paper are applicable to any metric space homeomorphic to $\mathbb{R}^2$, including fractal spaces, all of these examples have locally finite Hausdorff 2-measure. The four examples are summarized here, listed by section in which they appear.

\begin{enumerate}
    \item[5.1.] A surface that admits an MQC parametrization by $\mathbb{R}^2$ but not an \mbox{I-QS} parametrization. This surface is linearly locally connected (LLC) but not Loewner. This example also illustrates how metric quasiconformality is not preserved under taking inverses or precomposing with a QS map.
    \item[5.2.] A surface that admits a geometrically quasiconformal (QC) parametrization by $\mathbb{R}^2$ but not a MQC parametrization. This surface is upper Ahlfors 2-regular but not infinitesimally LLC.
    \item[5.3.] A surface that admits an I-QS parametrization by $\mathbb{R}^2$ but not a quasisymmetric parametrization. This surface is upper Ahlfors 2-regular but not LLC.
    \item[5.4.] A surface that, despite being a geodesic space of locally finite Hausdorff 2-measure, violates infinitesimal upper Ahlfors 2-regularity at every point along a nondegenerate continuum. This surface is LLC, and it admits a QC parametrization by $\mathbb{R}^2$ but not a MQC parametrization.
\end{enumerate}  
In particular, these examples show that the class of I-QS maps from $\mathbb{R}^2$ onto a metric space differs from both the class of QS maps and the class of MQC maps.

\section{$\mu$-quasiconformal maps} \label{sec:notation} 

We assume throughout the paper that $(X,d)$ is a metric space homeomorphic to the Euclidean plane $\R^2$. 
We denote $B(x,r)=\{y \in X: \, d(x,y)<r\}$, $\overline B(x,r) =\{y \in X: \, d(x,y)\leq r\}$, and $S(x,r)=\{y \in X: \, d(x,y)=r\}$. If $B$ is a ball of radius $r$, we denote by $\lambda B$ the ball with the same center and radius $\lambda r$. 
A \emph{path} in $X$ is a continuous map $\gamma:I \to X$, where $I$ is an interval. The image of such a path is called a \emph{curve} in $X$.

We recall the Cara\-th\'eo\-do\-ry construction of measures, cf. \cite[2.10]{Fed:69}. Let $\mathcal{F}$ be a family of subsets of $X$, and 
$\varphi:\mathcal{F}\to [0,\infty]$. For $A \subset X$ and $\delta>0$, the \emph{$\delta$-content} $\phi_\delta(A)$ is 
$$
\phi_\delta(A)=\inf \sum_{S \in \mathcal{G}} \varphi(S), 
$$
where the infimum is taken over all countable 
$$
\mathcal{G} \subset \{S \in \mathcal{F}: \, \diam(S) \leq \delta\} \quad \text{such that} \quad A \subset \bigcup_{S \in \mathcal{G}} S.   
$$ 
Then, since $\phi_\delta(A)$ is decreasing with respect to $\delta$, the limit 
$$
\psi(A)=\lim_{\delta \to 0^+} \phi_\delta(A) \in [0,\infty]
$$ 
exists. Moreover, if every $S\in \mathcal{F}$ is a Borel set, then $\psi$ is a Borel regular measure in $X$. 

Applying the Cara\-th\'eo\-do\-ry construction with $\mathcal{F}$ all the non-empty subsets of $X$ and $\phi(S)=\alpha(m)2^{-m}\diam(S)^m$ gives the \emph{$m$-dimensional Hausdorff measure} $\mathcal{H}^m$ in $X$, where $\alpha(1)=2$ and $\alpha(2)=\pi$.


Before defining $\mu$-quasiconformal maps, we review the classical geometric definition of quasiconformality. However, we replace the standard 
modulus of path families with the modulus of curve families, which lead to equivalent definitions but are easier to work with in our setting.  

Let $\Gamma$ be a family of curves (i.e., images of paths) in $X$. A Borel function $\rho \colon X \to [0,\infty]$ is \emph{admissible} for $\Gamma$ if $\int_{\mathcal{C}} \rho\, d\mathcal{H}^1 \geq 1$ for all 
$\mathcal{C} \in \Gamma$ with locally finite $\mathcal{H}^1$-measure. The \emph{(conformal) modulus} of $\Gamma$ is defined as
\begin{equation}
\label{eq:davycr}
\Mod \Gamma = \inf \int_X \rho^2\,d\mathcal{H}^2,
\end{equation}
where the infimum is taken over all admissible functions $\rho$. 

Let $X,Y$ be metric spaces homeomorphic to $\mathbb{R}^2$ and $f\colon X \to Y$ a homeomorphism. Then $f$ is \emph{geometrically quasiconformal} (QC), if there exists $K \geq 1$ such that
\[K^{-1} \Mod \Gamma \leq \Mod f\Gamma \leq K \Mod \Gamma\]
for all curve families $\Gamma$ in $X$. In this case, we also say that $f$ is \emph{geometrically  $K$-quasiconformal} ($K$-QC). 


We now define $\mu$-quasiconformal maps. Let $\mu$ be a Radon measure in $X$ with no atoms such that $\mu(B)>0$ for every open ball $B \subset X$. Recall that a Borel regular measure $\mu$ is \emph{Radon} if it is finite on compact sets. 

We associate with $\mu$ a collection $\mathcal{B}$ of open balls in $X$ such that for every point $x \in X$ there is $r_x>0$ such that $B(x,r) \in \mathcal{B}$ for every $r < r_x$. We also make the requirement that $\overline{B}(x,r_x)$ is compact for all $x$. 
We refer to such a collection $\mathcal{B}$ as an \emph{admissible cover}. From now on we use the convention that every measure $\mu$ comes equipped with an admissible cover $\mathcal{B}$. 

\begin{definition}
The \emph{$\mu$-length measure} $\ell_\mu$ in $X$ is defined by the Cara\-th\'eo\-do\-ry construction with $\mathcal{F}=\mathcal{B}$ and $\varphi\colon \mathcal{B} \to [0, \infty]$, $\varphi(B)=2\pi^{-1/2} \mu(B)^{1/2}$. 
\end{definition}

The $\ell_\mu$ is normalized so that if $X=\mathbb{R}^2$ and $\mu$ the Lebesgue measure, then $\ell_\mu=\Ha^1$ (for any choice of $\mathcal B$). 

\begin{definition}
\label{defn:lise}
Let $\Gamma$ be a family of curves in $X$. We say that a Borel function $\rho \colon X \to [0,\infty]$ is \emph{$\mu$-admissible} for $\Gamma$ if 
$\int_{\mathcal C} \rho \, d\ell_\mu \geq 1$ for all $\mathcal C \in \Gamma$ with locally finite $\ell_\mu$-measure. 
We denote the set of such functions by $\Phi_\mu(\Gamma)$. The \emph{$\mu$-modulus} of $\Gamma$ is 
\[
\Mod_\mu \Gamma =\inf_{\rho \in \Phi_\mu(\Gamma)} \int_X \rho^2 \, d\mu. 
\]
\end{definition}

Notice that if $\ell_\mu(\mathcal{C})=0$ for some $\mathcal{C} \in \Gamma$, then there are no $\mu$-admissible functions for $\Gamma$ and thus $\Mod_\mu \Gamma =\infty$. On the other hand, if $\ell_\mu$ is not locally finite on any $\mathcal{C} \in \Gamma$, then $\Mod_\mu \Gamma =0$. \Cref{defn:lise} coincides with \eqref{eq:davycr} when $X=\R^2$ and $\mu$ the Lebesgue measure. 

\begin{definition} \label{def:muqc}
Let $f \colon X \to \Omega$ be a homeomorphism, where $\Omega$ is a domain in $\R^2$. We say that $f$ and $f^{-1}$ are \emph{$\mu$-quasiconformal}, if there exists $K \geq 1$ such that  
\[
K^{-1} \Mod_\mu \Gamma  \leq \Mod f\Gamma  \leq K \Mod_\mu \Gamma  
\] 
for every curve family $\Gamma$ in $X$. 
\end{definition}

\Cref{def:muqc} naturally leads to the following questions: 
\begin{enumerate}
\item How to decide if a given metric space $X$ carries a measure $\mu$ for which there exists a $\mu$-quasiconformal map into $\R^2$? 
\item How to decide if there exists a $\mu$-quasiconformal map for a given $(X,\mu)$? 
\end{enumerate}
Concerning Question (2), it is reasonable to ask if the \emph{reciprocality} condition (\Cref{def:recip} below) can be modified to yield a characterization similar to the one obtained in \cite{Raj:16} for the $2$-dimensional Hausdorff measure. In the next section we introduce infinitesimally metric measures and show that they lead to the existence of $\mu$-quasiconformal maps.


\section{Infinitesimally metric measures} \label{sec:proof_mu_quasiconformal}

We now define the infinitesimally metric measures. Let $X$, $\mu$, $\mathcal{B}$ and $\ell_\mu$ be as above. Moreover, for $x,y \in X$ let 
\begin{equation*}
q(x,y)=\inf \ell_\mu(\mathcal{C}(x,y)), 
\end{equation*}
where the infimum is taken over all curves $\mathcal{C}(x,y)$ that join $x$ and $y$ in $X$. Thus $q$ defines a pseudometric on $X$. In the following, we use the subscripts $d$ and $q$ to indicate which (pseudo)metric is being used in our notation for balls, spheres, and diameter. 

\begin{definition} \label{def:imm}
The measure $\mu$ is \emph{infinitesimally metric} (I-MM) if there exist $\Lambda>1$, $C_i \ge 1$ such that 
\begin{equation}
\label{equ:quumetric} 
C_i^{-1}q(y,z) \leq \mu(B_d(x,r))^{1/2} \leq C_i q(y,z) 
\end{equation} 
for every $B_d(x,r) \in \mathcal B$, $y \in B_d(x,r/\Lambda)$ and $z \in S_d(x,r)$.
\end{definition}

It follows immediately from the definition that if $\mu$ is I-MM, then $q$ is a metric on $X$. 

Recall that a metric space $X$ is \emph{(Ahlfors) $2$-regular} if there exists $C\geq 1$ such that $C^{-1}r^2 \leq \mathcal{H}^2(B(x,r)) \leq Cr^2$ for all $x \in X$, $r \in (0, \diam X)$. We say that $X$ is \emph{lower} or \emph{upper $2$-regular} if, respectively, the first or second of these inequalities holds. \Cref{def:imm} imposes a similar infinitesimal condition on the measure $\mu$. In fact, we show in \Cref{ahlfors} and \Cref{mumass} that $(X,q)$ is infinitesimally Ahlfors 2-regular. 


The remainder of this section is dedicated to the proof of \Cref{thm:mu_quasiconformal}. We first restate the theorem. 
\begin{theorem} \label{thm:mu_quasiconformal'}
Let $X$ be a metric space homeomorphic to $\mathbb{R}^2$ which supports an I-MM $\mu$. Then there exists a $\mu$\-/quasiconformal map $f \colon X \to \Omega$, 
where $\Omega=\mathbb{D}\subset \mathbb{R}^2$ or $\Omega=\mathbb{R}^2$. 
\end{theorem}

As groundwork, we require several lemmas to estimate the 1- and 2-dimensional Hausdorff measures corresponding to the metric $q$. 

We fix an I-MM $\mu$. Let $\mathcal B =\{B_d(x,r): x \in X, r<r_x\}$ be the admissible cover associated with $\mu$.
The assumption that $\mu$ has no atoms implies that $\lim_{r \rightarrow 0} \mu(B_d(x,r)) = 0$ for all $x \in X$. \Cref{def:imm} then implies that the metrics $d$ and $q$ are topologically equivalent. 

\begin{lemma} \label{lemm:closedball}
	We have 	\[\mu(\overline B_d(x,r)) \le C_i^2 \mu(B_d(x,r))\]
	for every $B_d(x,r) \in \mathcal B$, where $C_i$ is the constant in \Cref{def:imm}.
\end{lemma}
\begin{proof}
	Since $\overline{B}_d(x,r)$ is compact and $X$ homeomorphic to $\R^2$, there exists a point $z \in \partial (X \setminus \overline{B}_d(x,r))$. Observe that $z \in S_d(x,r)$. Let $\e>0$, and let $w \in B_q(z,\e)$ such that $r<d(x,w)<r_x$. Now,
	\begin{align*}
	\mu(\overline B_d(x,r))^{1/2} &\le \mu(B_d(x,d(x,w))^{1/2} \\
	&\le C_iq(x,w) \le C_iq(x,z)+C_i\varepsilon\\
	&\le C_i^2 \mu(B_d(x,r))^{1/2}+C_i\e.
	\end{align*}
	Letting $\varepsilon\to 0$ proves the claim.
\end{proof}

\begin{lemma} \label{ahlfors}
	We have 
	\[C_i^{-2} r^2 \le \mu(B_q(x,r)) \le C_i^3 r^2\]
	for every ball $B_q(x,r)$ contained in $B_d(x,r_x/2)$, where $C_i$ is the constant in \Cref{def:imm}. 
\end{lemma}
\begin{proof}
	Let
	\[s = \inf_{y \in X \setminus B_q(x,r)} d(x,y) \,\,\, \text{ and } \,\,\, t = \sup_{z \in B_q(x,r)} d(x,z).\]
	Clearly $B_d(x,s) \subset B_q(x,r)$.  We claim that there exists $y \in S_d(x,s)$ such that $q(x,y) \geq r$. If not, then $X \setminus B_q(x,r)$ and $\overline{B}_d(x,s)$ are disjoint closed sets, with $\overline{B}_d(x,s)$ compact. This implies that $\dist(X \setminus B_q(x,r), \overline{B}_d(x,s))>0$, contradicting the definition of $s$. Since $\mu$ is assumed to be I-MM, we have $\mu(B_q(x,r)) \geq \mu(B_d(x,s)) \geq C_i^{-2}r^2$. 
		
	Likewise, $B_q(x,r) \subset \overline B_d(x,t)$. Similarly to the first part of the proof, we note that $(X \setminus B_d(x,t)) \cap \overline{B}_q(x,r) \neq \emptyset$. Thus, there exists $z \in S_d(x,t)$ such that $q(x,z) \leq r$. Since $\mu$ is I-MM, \Cref{lemm:closedball} gives
	\[ \mu(B_q(x,r)) \leq \mu(\overline{B}_d(x,t)) \leq C_i^2 \mu(B_d(x,t))\leq C_i^3r^2. \qedhere\] 
\end{proof}

For $s,\delta>0$, let $\Ha_q^s$ and $\Ha_{q,\delta}^s$ denote the $s$-dimensional Hausdorff measure and Hausdorff $\delta$-content on $(X,q)$, respectively. 

\begin{lemma} \label{mumass}
	We have 
	\[\frac{\pi}{4C_i^2} \mu(A) \le \Ha_q^2(A) \le 100\pi C_i^2\mu(A)\]
	for any Borel set $A \subset X$, where $C_i$ is the constant in \Cref{def:imm}. 
\end{lemma}
\begin{proof}
	Let $\delta>0$, and let $U\subset X$ be an open set with $A \subset U$ and $\mu(U) \le \mu(A) + \delta$. Using the basic covering theorem (see \cite[Thm. 1.2]{Hei:01}), choose a sequence of pairwise disjoint balls $B_j =B_q(x_j,r_j)$ with $B_j \subset U$, $B_j \subset B_d(x_j,r_{x_j}/2)$ and $10r_j<\delta $ for all $j$, such that $U \subset \cup_{j=1}^\infty 5B_j$. Then
	\[ \Ha_{q,\delta}^2(A) \le \pi \sum_{j=1}^{\infty} (10r_j)^2 \le C\pi \sum_{j=1}^\infty \mu(B_j) \le C\pi \mu(U) \le C\pi (\mu(A)+\delta),\]
	where $C=100C_i^2$ (the $\pi$ comes from the normalization of $\mathcal{H}^2$). The upper bound for $\Ha_q^2(A)$ follows. 
	
	For the lower bound, fix $n$ and define the Borel set
	\[A_n= \overline{\{ x \in A: B_q(x,1/n) \subset B_d(x,r_x/2) \} } \cap A.\] 
	Let $\{E_j\}$ be a cover for $A_n$ with $\diam_q(E_j) < \frac{1}{2n}$ for all $j$. Removing sets from the cover if necessary, we may assume that for every $j$ there exists $x_j \in A_n$ such that $E_j \subset B_q(x_j,2\diam_q E_j)$ and $B_q(x_j,1/n) \subset B_d(x_j,r_{x_j}/2)$. Since
	\[\mu(A_n) \le \sum_{j=1}^\infty \mu(B_q(x_j,2\diam_q E_j)) \le 4C_i^2 \sum_{j=1}^\infty \diam_q(E_j)^2,\]
	we get
	\[ \frac{\pi}{4C_i^2} \mu(A_n) \le \Ha_{q,1/2n}^2(A_n) \le \Ha_q^2(A).\]
	Since $\mu(A)=\lim_{n\to\infty}\mu(A_n)$, the claim follows.
\end{proof}

\begin{lemma} \label{mulength}
	We have 
	\[\frac{2}{C_i\sqrt{\pi}} \Ha_q^1(A) \le \ell_\mu(A) \le \frac{4 C_i^3}{\sqrt{\pi}}\Ha_q^1(A) \]
	for any Borel set $A \subset X$, where $C_i$ is the constant in \Cref{def:imm}. 
\end{lemma}

\begin{proof}
	Since $X$ is homeomorphic to $\R^2$, it is locally compact and can be exhausted by compact sets $X^j$. We can also approximate both $\ell_\mu(A)$ and $\Ha^1_q(A)$ from below with the measures of the sets $A^j=A \cap X^j$, and by considering some compact neighbourhood $X^{j+k}$ of $A^j$ we can assume that
	\[ \sup_{x \in X} \diam_q(B_d(x,r)), \, \sup_{x \in X} \diam_d(B_q(x,r)) \to 0 \, \text{ as } r \to 0.\]
	
	We first consider Borel sets 
	\[A_n = \overline{ \{ x \in A: 1/n < r_x \} } \cap A, \, n \in \N.\]
	Let $\sigma>0$ be arbitrary and $\delta >0$ small enough so that $\diam_d (B_q(x,2\delta)) < \min \{\sigma, 1/n \}$ for every $x$. Fix any cover $\{E_j\}$ of $A_n$ with $\diam_q(E_j) < \delta$ for all $j$. Removing sets from the cover if necessary, we may assume that for every $j$ there exists $x_j \in A_n$ such that $\dist(\{x_j\},E_j) < \diam_q(E_j)$ and $r_{x_j} > 1/n$. Let
	\[t_j = \inf \{ t>0: E_j \subset B_d(x_j,t) \}.\]
	Then for every $j$ we have $E_j \subset \overline B_d(x_j,t_j)$. Moreover, since \[E_j \subset B_q(x_j,2\diam_q E_j) \subset B_q(x_j,2\delta),\] we have 
	$t_j < \min \{ \sigma, 1/n \}$. 
	
	For every $j, m \in \N$ there exists $y_m^j \in E_j \setminus B_d(x_j,t_j-1/m)$, so that 
	\[  \mu(B_d(x_j,t_j-1/m))^{1/2} \le C_i q(x_j,y_m^j).\]

	Since $y_m^j \in B_q(x_j,2\diam_q E_j),$ we have
	$\mu(B_d(x_j,t_j))^{1/2} \le 2 C_i \diam_q(E_j)$. Recall that $\ell_\mu$ is defined by the Carath\'eodory construction: 
	$\ell_\mu(A_n)=\lim_{\sigma \to 0} \ell_{\mu,\sigma}(A_n)$, where 
	$\ell_{\mu,\sigma}$ is the corresponding $\sigma$-content. By \Cref{lemm:closedball} we get
	\begin{eqnarray*} 
	\ell_{\mu,\sigma}(A_n) &\le& 2\pi^{-1/2} \sum_j \mu(\overline B_d(x_j,t_j))^{1/2} \le 2\pi^{-1/2} C_i^2 \sum_j \mu(B_d(x_j,t_j))^{1/2}\\  &\le& 
	4  \pi^{-1/2} C_i^3 \sum_j \diam_q(E_j) 
	\end{eqnarray*}
	(the $2\pi^{-1/2}$ comes from the normalization of $\ell_\mu$) and hence $\ell_{\mu,\sigma}(A_n) \le 2\pi^{-1/2}C_i^3 \mathcal{H}^1_q(A_n)$. This holds for all $\sigma>0$ and 
	$n \in \N$, so we have $\ell_\mu(A) \le 4 \pi^{-1/2}C_i^3\Ha^1_q(A)$.
	
	The other inequality can be proved more directly, with similar arguments but without the need to consider the sets $A_n$.
\end{proof}

We will apply the main result in \cite{Raj:16}. It depends on the following definition. A \emph{quadrilateral} $Q=Q(\zeta_1, \zeta_2, \zeta_3, \zeta_4)$ is a set homeomorphic to a closed square in $\R^2$, with boundary edges $\zeta_1, \zeta_2, \zeta_3, \zeta_4$ (in cyclic order). For sets $E,F \subset G$, $\Gamma(E,F;G)$ denotes the family of curves in $G$ that join 
$E$ and $F$. While path families were considered in \cite{Raj:16}, the results applied below remain valid when they are replaced with curve families. 

\begin{definition} \label{def:recip}
	Let $Y$ be a metric space homeomorphic to $\R^2$ with locally finite Hausdorff 2-measure. The space $Y$ is \emph{reciprocal} if there exists $\kappa \geq 1$ such that for all quadrilaterals $Q = Q(\zeta_1, \zeta_2, \zeta_3, \zeta_4)$ in $X$,
	\begin{equation} \label{equ:reciprocality(1)}
	\Mod \Gamma(\zeta_1, \zeta_3; Q) \Mod \Gamma(\zeta_2, \zeta_4; Q) \leq \kappa
	\end{equation}
	and for all $x \in X$ and $R>0$ such that $X \setminus B(x,R) \neq \emptyset$,
	\begin{equation} \label{equ:reciprocality(3)} 
	\lim_{r \rightarrow 0} \Mod \Gamma(B(x,r), X \setminus B(x,R); B(x,R)) = 0.
	\end{equation} 
\end{definition}

It was shown in \cite{RajRom:19} that the inequality opposite to \eqref{equ:reciprocality(1)} holds in every $Y$. That is, there exists a universal constant $\kappa'>0$ such that \[\Mod \Gamma(\zeta_1, \zeta_3; Q) \Mod \Gamma(\zeta_2, \zeta_4; Q) \geq \kappa'\] for all quadrilaterals $Q \subset Y$. 


\begin{theorem}[Theorem 1.4 \cite{Raj:16}] \label{RtoQC}
	Let $Y$ be a metric space homeomorphic to $\R^2$, with locally finite Hausdorff 2-measure. There exists a QC map $h \colon Y \to \Omega \subset \R^2$ if and only if $Y$ is reciprocal.
\end{theorem}

The next proposition is a generalization of Theorem 1.6 from \cite{Raj:16}, where the mass upper bound is assumed for every radius. 

\begin{proposition} \label{prop:reciprocality_criterion}
	Let $Y$ be a metric space homeomorphic to $\R^2$. Suppose there exist $C_U>0$ and for every $y\in Y$ a radius $r_y>0$ such that
	\begin{equation} \label{equ:upper_ahlfors}\Ha^2(B(y,r)) \le C_U r^2 \end{equation}
	for every $r<r_y$. Then $Y$ is reciprocal.
\end{proposition}

\begin{proof} 
Condition \eqref{equ:reciprocality(3)} follows by considering the admissible function 
\[\rho(z)=\frac{1}{\log(R/r) d(y,z)}.\] 
To prove \eqref{equ:reciprocality(1)}, we modify the proof of \cite[Proposition 15.5]{Raj:16}. We give the main steps and refer to \cite{Raj:16} for the missing details. 
Let $Q=Q(\zeta_1, \zeta_2, \zeta_3, \zeta_4)$ be a quadrilateral. Then there exists a $\rho$ that is \emph{weakly admissible} (admissible outside an exceptional curve family of zero modulus) for $\Gamma(\zeta_1,\zeta_3;Q)$, such that 
\[\int_Y \rho^2 \, d\mathcal{H}^2 =\Mod \Gamma(\zeta_1, \zeta_3; Q).\]
Fix a curve $\mathcal{C} \in \Gamma(\zeta_2,\zeta_4;Q)$. We may assume that $\mathcal{C}$ is homeomorphic to $[0,1]$ and has finite length. 
Using the basic covering theorem, we find a finite cover $\{5B_j\}=\{B(y_j,5r_j)\}$ of $\mathcal{C}$ such that $y_j \in \mathcal{C}$ and $36r_j<r_y$ for all $j$, and such that the balls 
$B_j$ are pairwise disjoint. Moreover, let $g\colon Q \to [0,\infty]$, 
\begin{equation}
\label{equ:minimi}
g(y)=\sum_j r_j^{-1} \chi_{6B_j\cap Q}(y). 
\end{equation}
Since every $\mathcal{C}'$ in $\Gamma(\zeta_1,\zeta_3;Q)$ intersects at least one of the balls $5B_j$, it follows that $g$ is admissible for $\Gamma(\zeta_1,\zeta_3;Q)$. Moreover, since $\rho$ is a minimizer for $\Mod \Gamma(\zeta_1, \zeta_3; Q)$, applying the weak admissibility of $(1-t)\rho+tg$ and letting $t \to 0$ leads to 
\begin{equation}
\label{equ:uomi}
\Mod \Gamma(\zeta_1,\zeta_3;Q) \leq \int_Q \rho g \, d\mathcal{H}^2 = \sum_j r_j^{-1} \int_{6B_j\cap Q} \rho \, d\mathcal{H}^2. 
\end{equation}
For the maximal function $\mathcal{M}\rho:Q \to [0,\infty$], 
\[
\mathcal{M}\rho(z)=\sup_{r>0} \frac{1}{\mathcal{H}^2(B(z,5r))} \int_{B(z,r)\cap Q} \rho \, d\mathcal{H}^2, 
\]
standard arguments show that  
\begin{equation}
\label{equ:kissapoyta}
\int_Q (\mathcal{M}\rho)^2 \, d\mathcal{H}^2 \leq 8 \int_Q \rho^2 \, d\mathcal{H}^2.  
\end{equation}
Now we apply \eqref{equ:upper_ahlfors} to estimate the right hand term of \eqref{equ:uomi} from above by 
\begin{align*} 
& 1296C_U\sum_j \frac{r_j}{\mathcal{H}^2(B(y_j,36B_j)} \int_{B(y_j,6_j)\cap Q} \rho \, d\mathcal{H}^2 \\ & \leq 1296C_U\sum_j r_j \inf_{y \in \mathcal{C} \cap B_j} \mathcal{M}\rho(y). 
\end{align*} 
Since the right hand term is bounded from above by $1296C_U \int_\mathcal{C} \mathcal{M}\rho \, d\mathcal{H}^1$, we conclude that 
\[
y \mapsto \frac{1296C_U\mathcal{M}\rho(y)}{\Mod \Gamma(\zeta_1,\zeta_3;Q)} 
\]
is admissible for $\Gamma(\zeta_2,\zeta_4;Q)$. Combining the admissibility with \eqref{equ:minimi} and \eqref{equ:kissapoyta}, we have 
\[
\Mod \Gamma(\zeta_2,\zeta_4;Q) \leq \frac{8 \cdot 1296^2 C_U^2}{\Mod \Gamma(\zeta_1,\zeta_3;Q)}, 
\]
from which \eqref{equ:reciprocality(3)} follows. 
\end{proof}

\begin{proof}[Proof of \Cref{thm:mu_quasiconformal}]
	By Lemmas \ref{ahlfors} and \ref{mumass}, the space $(X,q)$ satisfies the assumption of \Cref{prop:reciprocality_criterion}. Thus by \Cref{RtoQC} there exists a QC map $h \colon (X,q) \to \Omega \subset \R^2$. By the Riemann mapping theorem, we can choose $h$ such that $\Omega=\mathbb{D}$ or $\Omega=\R^2$. Moreover, by Lemmas \ref{mumass} and \ref{mulength} the $\mu$-modulus $\Mod_\mu(\Gamma)$ and the conformal 2-modulus $\Mod_2(\Gamma)$ in $X$ are comparable for any curve family $\Gamma$, so $h$ precomposed with the identity map from $(X,d)$ to $(X,q)$ is $\mu$-quasiconformal. 
\end{proof}


\section{Infinitesimally quasisymmetric maps} \label{sec:proof_iqs}

In this section we introduce the notion of infinitesimally quasisymmetric map and apply our results on infinitesimally metric measures to give a characterization for the spaces that admit such a parametrization by a Euclidean planar domain. 

Recall that a homeomorphism $f \colon (X,d) \to (Y,d')$ between metric spaces is \emph{quasisymmetric} (QS) if there exists a homeomorphism $\eta \colon [0,\infty) \to [0, \infty)$ such that 
\begin{equation}
\label{equ:qsu}
\frac{d(x,y)}{d(x,z)} \le t \, \text{ implies } \, \frac{d'(f(x),f(y))}{d'(f(x),f(z))} \le \eta(t) 
\end{equation}
for all distinct points $x, y, z \in X$. Closely related is the following definition. A homeomorphism $f\colon (X,d) \to (Y,d')$ between metric spaces is \emph{metrically quasiconformal} (MQC) if there exists $H \geq 1$ such that 
\[\limsup_{r \to 0} \frac{\sup\{d'(f(x),f(y)): d(x,y) \leq r\}}{\inf\{d'(f(x),f(y)): d(x,y) \geq r\}} \leq H \]
for all $x \in X$. 

\begin{definition} \label{def:qs}
A homeomorphism $f \colon (X,d) \to (Y,d')$ is \emph{infinitesimally quasisymmetric} (I-QS) if there exists a homeomorphism $\eta \colon [0,\infty) \to [0, \infty)$ such that for every $x \in X$ there exists a radius $r_x>0$ such that \eqref{equ:qsu} holds for all $y,z \in B(x,r_x)$. 
\end{definition}

It is a standard exercise to show that if $f\colon X \to Y$ and $g\colon Y \to Z$ are QS, then $g \circ f$ and $f^{-1}$ are also QS. These properties also hold for the class of I-QS maps. Note that both properties may fail for MQC maps, even for metric spaces homeomorphic to $\mathbb{R}^2$. In \Cref{sec:exponential_weight}, we give an example of this.

It is immediate from the definitions that any QS map is I-QS, and any I-QS is MQC. Thus infinitesimal quasisymmetry is an intermediate condition between quasisymmetry and metric quasiconformality. In \Cref{sec:spikes_2}, we give an example of a map which is I-QS but not QS. 

Recall that a metric space $(X,d)$ is \emph{linearly locally connected} (LLC) if there exists $\lambda \geq 1$ such that the following properties hold: \begin{enumerate}
\item For any $x \in X$, $r>0$ and $y,z \in B(x,r)$ there exists a continuum $K \subset B(x,\lambda r)$ with $y,z \in K$. 
\item For any $x \in X$, $r>0$ and $y,z \in X \setminus B(x,r)$ there exists a continuum $K \subset X \setminus B(x,\lambda^{-1}r)$ with $y,z \in K$.
\end{enumerate}
\begin{definition}
\label{def:illc}
A metric space $(X,d)$ is \emph{infinitesimally linearly locally connected} (I-LLC) if there exists $\Lambda \geq 1$ such that for every $x \in X$ there exists a radius $r_x>0$ such that the above properties hold for all $r<r_x$.
\end{definition}

It is easy to see that the LLC property is preserved under QS maps. Similarly, I-QS maps preserve the I-LLC property. 
Since every planar domain is I-LLC, any metric space that admits an I-QS map to such a domain must also be I-LLC. 

Finally, we introduce a modification of the Loewner condition of Heinonen and Koskela \cite{HK:98}. We denote by $\Gamma(A,B)$ the family of curves which join sets $A$ and $B$ in $X$. 
Recall that $X$ (equipped with $\mathcal{H}^2$) is \emph{Loewner} if there exists a decreasing function $\phi\colon (0,\infty) \to (0,\infty)$ such that $\Mod \Gamma(E,F) \geq \phi(t)$ 
for all disjoint nondegenerate continua $E,F$ satisfying 
\begin{equation}
\label{equ:loew}
\frac{\dist(E,F)}{\min\{\diam E, \diam F\}} \leq t. 
\end{equation}
Also, recall our convention that any measure $\mu$ comes equipped with an admissible cover $\mathcal{B}=\{B(x,r): \, 0<r<r_x\}$. 

\begin{definition}
\label{def:iloew}
A metric space $X$ equipped with a measure $\mu$ is \emph{infinitesimally Loewner} (I-Loewner) if there exists a decreasing function $\phi \colon (0,\infty) \to (0,\infty)$ such that $\Mod_\mu \Gamma(E,F) \geq \phi(T)$ for all disjoint continua $E,F$ such that $E$ joins $x$ and $S(x,t)$, $F\supset S(x,r_x)$ joins $S(x,s)$ and $S(x,r_x)$, and $0<s,t<r_x/2$, $s/t \leq T$. 
\end{definition}

It follows from the Loewner property of $\mathbb{R}^2$ that every planar domain, equipped with Lebesgue measure and any admissible cover, 
is I-Loewner. The remainder of this section is dedicated to the proof of \Cref{thm:iqs}. We first restate the theorem. 

\begin{theorem} \label{thm:iqs'}
Let $X$ be a metric space homeomorphic to $\mathbb{R}^2$. There exists an I-QS 
map $f\colon X \to \Omega$, where $\Omega=\mathbb{D}$ or $\Omega=\R^2$, if and only if $X$ is I-LLC
and supports an I-MM $\mu$ such that $(X,\mu)$ is I-Loewner. 
\end{theorem}

To prove the theorem, we first show in \Cref{mulength2} and \Cref{baqaa} that if $f\colon X \to \Omega$ is I-QS, then the pullback of Lebesgue measure satisfies the conditions of the theorem (we already noticed that the existence of $f$ forces $X$ to be I-LLC). For the other direction, we show in \Cref{keqaa} that $\mu$-quasiconformal maps $X \to \Omega \subset \R^2$, such as the map in \Cref{thm:mu_quasiconformal}, are I-QS under these conditions. \Cref{keqaa} can be seen as an infinitesimal analog of \cite[Theorem 4.7]{HK:98}, and it is proved using similar arguments.  

\begin{lemma} \label{mulength2}
	
	Let $f \colon X \to \Omega \subset \R^2$ be an I-QS map, and $\mu = f^*\mathcal{L}_2$ the pullback measure of the Lebesgue measure $\mathcal{L}_2$. Equip $\mu$ with admissible cover $\mathcal{B}=\{B(x,r): \, 0<r<r_x\}$, where the $r_x$ are the radii in \Cref{def:qs} of I-QS maps. Then 
	\[\eta(1)^{-1} \Ha^1(f(\mathcal C)) \le \ell_\mu(\mathcal C) \le 4\eta(5) \Ha^1(f(\mathcal C))\]
	for any curve $\mathcal C \subset X$.
	
\end{lemma}

\begin{proof}
	
	We may assume that the curve $\mathcal C$ is simple and compact. As in \Cref{mumass}, it suffices to prove the claim for sets $\mathcal{C}$ for which there exists 
	$\delta>0$ such that the set of points $x$ satisfying $B(x,\delta) \in \mathcal B$ is dense in $\mathcal C$.
	
	Fix such a $\delta$ and a sequence $(B_j)=(B(x_j,r_j))$ of disjoint balls such that $x_j \in \mathcal C$, $5B_j \in \mathcal B$, $5r_j < \delta$ and $\mathcal C \subset \cup_j 5B_j$, ordered so that if $\gamma$ is any injective parametrization of $\mathcal C$ then $s_j=\gamma^{-1}(x_j)$ is a monotone sequence.
	
	Let
	\[T_j=\sup \{ t>0 : B(f(x_j),t) \subset f(B_j) \}\]
	for every $j$. Then there exists $z_j \in X$ with $d(x_j,z_j) \ge r_j$ and $\abs{f(x_j)-f(z_j)} \le 2 T_j$. Using the infinitesimal quasisymmetry of $f$ we find that for any $y_j \in 5B_j$
	\[ \abs{f(x_j)-f(y_j)} \le \eta(5) \abs{f(x_j)-f(z_j)}\]
	so that $f(5B_j) \subset B(f(x),2\eta(5)T_j)$. By the choice of $T_j$ also $B(f(x),T_j) \subset f(B_j)$ and thus 
	\[\mu(5B_j)= \mathcal{L}_2(f(5B_j)) \le 4\pi \eta(5)^2T_j^2 \le 4\pi \eta(5)^2\abs{f(x_j)-f(x_k)}^2\] for all $k \ne j$ as the balls $B_j$ are disjoint. Now the $\delta$-content $\ell_{\mu,\delta}$ satisfies 
	\[\ell_{\mu,\delta}(\mathcal C) \le 2\pi^{-1/2}\sum_j \mu(5B_j)^{1/2} \le  4\eta(5)\sum_j \abs{f(x_j)-f(x_{j+1})}.\]
	Since $f(\mathcal C)$ is the nonoverlapping union of the subcurves connecting $f(x_j)$ and $f(x_{j+1})$, we have $\ell_{\mu,\delta}(\mathcal C) \le 4\eta(5) \Ha^1(f(\mathcal C))$ for any $\delta >0$ and thus $\ell_\mu(\mathcal C) \le 4\eta(5) \Ha^1(f(\mathcal C)).$
	
	To prove the other inequality, fix $\varepsilon>0$ and let $B_j=B(x_j,r_j)$ be a sequence of balls in $\mathcal B$ covering $\mathcal C$ with $\diam B_j < \sigma$ and $B_j \cap \mathcal C \ne \emptyset$ for all $j$ and some $\sigma >0$. Since $X$ is locally compact and $\mathcal C$ is compact, $\diam f(B_j) < \varepsilon$ for all $j$ when $\sigma $ is sufficiently small.
	
	By the infinitesimal quasisymmetry of $f$ we have 
	\[\diam f(B_j)^2 \le 4\pi^{-1}\eta(1)^2 \mathcal{L}_2(f(B_j))=4\pi^{-1}\eta(1)^2 \mu(B_j)\] for every $j$, and hence
	\[ \Ha^1_\varepsilon(f(\mathcal C)) \le  2\pi^{-1/2}\eta(1)\sum_j  \mu(B_j)^{1/2}.\]
	Thus $\Ha^1_\varepsilon(f(\mathcal C)) \le \eta(1) \ell_{\mu,\sigma}(\mathcal C) \le \eta(1)\ell_\mu(\mathcal C)$, and the same upper bound holds for $\Ha^1$ since $\varepsilon$ was arbitrary.
\end{proof}

\begin{corollary} \label{iqsismuqc}
Let $f$ and $\mu$ be as in \Cref{mulength2}. Then $f$ is \newline $\mu$-quasiconformal.
\end{corollary}

\begin{proof}
Let $\Gamma$ be a curve family in $X$ and $\e>0$. We choose a $\mu$-admissible function $\rho$ with $\int_X \rho^2 \, d \mu \le  \modd_\mu(\Gamma)+\e$ and define $\tilde \rho = \rho \circ f^{-1}$ in $\Omega$. If a curve $\mathcal C \in \Gamma$ has locally finite $\ell_\mu$-measure, then by \Cref{mulength2} and a change of variables
$$ \int_{f(\mathcal C)} \tilde \rho \, d \Ha^1 \ge \frac1{4\eta(5)} \int_{\mathcal C} \rho \, d \ell_\mu ,$$
so that $ 4\eta(5) \tilde \rho$ is admissible for $f(\Gamma)$. Thus using the definition of $\mu$ and a change of variables we have
$$ 
\modd(f(\Gamma)) \le 16\eta(5)^2 \int_\Omega \tilde \rho^2 \, d \mathcal L_2 = 16\eta(5)^2 \int_X \rho^2 \, d \mu \le 16\eta(5)^2 \left( \modd_\mu(\Gamma)+\e \right). 
$$
The other direction can be proved similarly using the other inequality of \Cref{mulength2}.
\end{proof}

\begin{proposition} \label{baqaa}
	Let $f$, $\mu$ and $\mathcal{B}$ be as in \Cref{mulength2}. Then $\mu$ is I-MM and satisfies the I-Loewner condition.  	
\end{proposition}

\begin{proof}
	
	Let $\Lambda >1$ be large enough so that $\eta(1/\Lambda) \le \frac12$. Fix $x \in X$ and $0<r<r_x/2$ so that $\overline{B}(f(x),\diam f(B(x,r))) \subset \Omega$. In order to prove the I-MM condition \eqref{equ:quumetric}, fix $y \in B(x,r/\Lambda)$ and $z \in S(x,r)$. Then the segment $[f(y),f(z)]$ is contained in 
	$\Omega$. Let $\mathcal C=f^{-1}([f(y),f(z)])$, which is a curve connecting $y$ and $z$.
	
	Now let 
	\[ T= \sup \{ t>0: B(f(x),t) \subset f(B(x,r)) \}.\] 
	Using \Cref{mulength2} and infinitesimal quasisymmetry, we have
	\begin{align*}
	\ell_\mu(\mathcal C) &\le 4\eta(5) \mathcal{H}^1(f(\mathcal C)) = 4\eta(5)\abs{f(y)-f(z)} \le 4\eta(5) \diam fB(x,r) \\ 
	&\le 8\eta(1)\eta(5) T \leq \frac{8\eta(1)\eta(5)}{\sqrt{\pi}} \mathcal{L}_2(f(B(x,r)))^{1/2} \\ & = \frac{8\eta(1)\eta(5)}{\sqrt{\pi}} \mu(B(x,r))^{1/2},
	\end{align*}
	so the first inequality in \eqref{equ:quumetric} holds. 
	
	For the reverse inequality, notice first that our choice of $\Lambda$ implies that $\abs{f(x)-f(y)}\le \frac12 \abs{f(x)-f(z)}$ and thus $\abs{f(y)-f(z)}\ge \frac12 \abs{f(x)-f(z)}$. Let $\mathcal C$ be any curve connecting $y$ and $z$. Now by \Cref{mulength2}
	\begin{align*}
	\ell_\mu(\mathcal C) & \ge \eta(1)^{-1} \mathcal{H}^1(f(\mathcal C)) \\
	& \ge \eta(1)^{-1} \abs{f(y)-f(z)}  \ge \frac{1}{2\eta(1)} \abs{f(x)-f(z)} \\
	& \ge \frac{1}{2\sqrt{\pi}\eta(1)^2} \mathcal{L}_2(f(B(x,r)))^{1/2} = \frac{1}{2\sqrt{\pi}\eta(1)^2} \mu(B(x,r))^{1/2},
	\end{align*}
	since $f(B(x,r)) \subset B(f(x),\eta(1)\abs{f(x)-f(z)})$. Hence also the second inequality in \eqref{equ:quumetric} holds. We conclude that $\mu$ is I-MM.
	
	Finally, we show the I-Loewner condition. Fix $x \in X$ and disjoint continua $E$ and $F$ as in \Cref{def:iloew}, so that there are 
	$y \in F \cap S(x,s)$ and $z \in E \cap S(x,t)$. By infinitesimal quasisymmetry, 
	\[\frac{\dist(fE,fF)}{\diam E}\leq \frac{\abs{f(y)-f(x)}}{\abs{f(z)-f(x)}}\leq \eta(s/t).\] 
	By definition, $F$ contains $S(x,r_x)$. In particular, $fS(x,r_x)$ surrounds $f(x)$, and we have $\dist(fE,fF) \leq \diam fF$. Combining the estimates yields  
	\[
	\frac{\dist(fE,fF)}{\min\{ \diam E, \diam F\}} \leq \max\{\eta(s/t),1\}. 
	\]
	Since $\mathbb{R}^2$ is Loewner, there is $\phi'$ such that 
	\[
	\Mod \Gamma(fE,fF) \geq \phi'(\max\{\eta(s/t),1\}). 
	\]
	On the other hand $f$ is $\mu$-quasiconformal by \Cref{thm:mu_quasiconformal}, so 
	\[
	\Mod_\mu \Gamma(E,F) \geq K^{-1}\Mod \Gamma(fE,fF) 
	\]
	for some $K \geq 1$. We conclude that the I-Loewner condition holds with $\phi(T)=K^{-1}\phi'(\max\{\eta(T),1\})$. 
	\end{proof}

\begin{proposition} \label{keqaa}
Let $\mu$ be an I-MM on $X$, and $f \colon X \to \Omega$ a $\mu$-quasi\-con\-for\-mal homeomorphism. Suppose that $X$ is I-LLC and $\mu$ satisfies the I-Loewner condition. Then $f$ is I-QS.
\end{proposition}

\begin{proof}
Let $\Lambda$ and $\lambda$ be the constants in Definitions \ref{def:imm} and \ref{def:illc} of I-MM and I-LLC, respectively. 
We will prove the equivalent statement that $g=f^{-1}$ is I-QS. In this proof, for a point $a \in \Omega$ and set $A \subset \Omega$, let $a'=g(a)$ and $A'=g(A)$. 

Fix $x \in \Omega$ and $r>0$ so that 
\[B(x,3r) \subset \Omega \cap g^{-1}\big(B(x',r_{x'}/(10\lambda^4 \Lambda^4))\big), \]
and $y,z \in B(x,r)$. By our choice of $r$, we can choose $w \in g^{-1}S(x',r_{x'})$ so that the segment $[x,w]$ contains $z$. Moreover, taking $r$ to be sufficiently small, we can ensure that the segment $[x,w]$ lies in $\Omega$. Notice that $w \notin B(x,3r)$. Let $m=d(x',y')$ and $ \ell = d(x',z')$. Let $t>0$. We must find an upper bound $\eta(t)$ on $m/\ell$ that holds whenever $\abs{x-y}/\abs{x-z} \leq t$, such that $\eta(t) \to 0$ as $t \to 0$. Assume then that $y,z$ satisfy $\abs{x-y}/\abs{x-z} \leq t$.

Suppose first that $m/\ell \ge \Lambda \lambda^2$. Then, by the I-LLC property, we can connect $x'$ to $z'$ by a continuum $E'$ contained in $B(x',\lambda \ell)$, and $y'$ to $w'$ by a continuum $F'$ contained in $X \setminus B(x',m/\lambda)$. Let $k = \lceil \log_\Lambda (m/(\ell \lambda^2)) \rceil$, 
\[B_j=B(x',\Lambda^{j} \ell/\lambda),\,\,\, \text{and} \,\,\, A_j=B(x',\Lambda^j \ell/\lambda) \setminus B(x',\Lambda^{j-1}\ell/\lambda).\] 
Then, by the definition of I-MM, 
\[\rho =\frac1k \sum_{j=1}^{k} \frac{C_i\chi_{A_j}}{\mu(B_j)^{1/2}} \]
is $\mu$-admissible for $\Gamma(E',F')$. Thus 
\[\Mod_\mu \Gamma(E',F') \le \int_X\rho^2 \, d\mu \le  \frac{1}{k^2}\sum_{j=1}^k \frac{C_i^2 \mu(A_j)}{\mu(B_j)} \leq \frac{C_i^2}{k}\leq \frac{C_i^2}{\log_\Lambda (m/(\ell \lambda^2))}. 
\]
Hence $\Mod_\mu \Gamma(E',F')$ becomes arbitrarily small as $m/\ell$ increases to infinity. 

Since $g$ is $\mu$-quasiconformal, $\Mod \Gamma(E,F)$ is also small, where $E=g^{-1}(E')$ and $F=g^{-1}(F')$. But these sets connect $x$ to $z$ and $y$ to $w$, respectively, and have relative distance
\[ \Delta(E,F) = \frac{\dist(E,F)}{\min \{ \diam E, \diam F \} } \le \frac{\abs{x-y}}{\abs{x-z}}.\]
Thus, by the Loewner property of $\R^2$, we have $\abs{x-y}/\abs{x-z} \to \infty$ as $m/\ell \to \infty$, establishing the distortion inequality in this case. 

Suppose then that $0<m/\ell < \Lambda \lambda^2$. In this case we choose $E=[x,y]$ and $F=[z,w]\cup g^{-1}S(x',r_{x'})$. 
We may assume that $2\abs{x-y}<\abs{x-z}$, since otherwise there is nothing to prove. Applying the I-Loewner condition to $E'$ and $F'$, we have   
\[ \Mod_\mu \Gamma(E',F') \geq \phi(\ell/m). \]

Combining with the $\mu$-quasiconformality of $g$, we get $\Mod \Gamma(E,F) \geq K^{-1}\phi(\ell/m)$. On the other hand, by our choice of $w$ we can estimate $\Mod \Gamma(E,F)$ from above as follows:  
\[ \Mod \Gamma(E,F) \leq \Mod \Gamma(S(x,\abs{x-z}),S(x,\abs{x-y}))= 2\pi \Big(\log \frac{\abs{x-z}}{\abs{x-y}}\Big)^{-1}. \]
Combining the estimates, we see that $\phi(\ell/m) \leq 2\pi K(\log(1/t))^{-1}$. Observe that this bound becomes arbitrarily small as $t \to 0$. Since $\phi$ is decreasing, this yields an upper bound $\eta(t)$ on $m/\ell$ that goes to zero as $t \to 0$.
\end{proof}


\section{Examples} \label{sec:exm}

In this section, we work out in detail a number of specific examples of metric spaces homeomorphic to the plane. All of our examples have locally finite Hausdorff 2-measure, and we assume throughout this section that a given metric space is equipped with the Hausdorff 2-measure. We write a point $x$ in coordinates as $x = (x_1, x_2)$ if $x \in \mathbb{R}^2$ or $x = (x_1, x_2, x_3)$ if $x \in \mathbb{R}^3$.



In addition to the examples of this section, we refer the reader to Example 4.7 of \cite{HK:95} for a family of uniformly LLC surfaces in $\mathbb{R}^3$, equipped with the ambient Euclidean metric, that are conformally equivalent but not uniformly QS equivalent to the Euclidean plane. 
We also refer to Example 2.1 of \cite{Raj:16} for an example of a non-reciprocal metric on the plane, and to Example 17.1 of \cite{Raj:16} for a non-rectifiable surface in $\mathbb{R}^3$ that is QC equivalent to the Euclidean plane. Finally, see \cite{Rom:19} for the construction of a surface of locally finite Hausdorff 2-measure that is QS equivalent to the plane but not QC equivalent.  


\subsection{Conformal weight that decreases rapidly near the origin} \label{sec:exponential_weight}

	Define a metric $d$ on the Riemann sphere $\widehat{\mathbb{R}}^2 = \mathbb{R}^2 \cup \{\infty\}$ via the conformal weight \[\omega(x) = \left\{ \begin{array}{cc} e^{-1/\abs{x}}/\abs{x}^2& \text{ if } x \neq 0 \\ 0 & \text{ if } x = 0, \infty \end{array} \right. .\]
	That is, for all $x,y \in \widehat{\mathbb{R}}^2$, the metric $d$ is given by $d(x,y) = \inf_\gamma \int_\gamma \omega\, ds$, where the infimum is taken over all absolutely continuous paths $\gamma\colon [0,1] \to \widehat{\mathbb{R}}^2$ such that $\gamma(0) = x$ and $\gamma(1) = y$. 
	
	It is easy to check that $d(0,x) = e^{-1/\abs{x}}$ for all $x \in \widehat{\mathbb{R}}^2 \setminus \{0\}$. In particular, $d(0,\infty) = 1$ and we see that $d$ is finite. Next, let $x,y \in \widehat{\mathbb{R}}^2 \setminus \{0\}$ and assume that $\abs{x} \leq \abs{y}$. By considering the concatenation of the straight-line path from $x$ to $(\abs{y} / \abs{x})x$ and a circular arc from $(\abs{y} / \abs{x})x$ to $y$, we obtain the estimate 
	\[d(x,y) \leq  e^{-1/\abs{y}} - e^{-1/\abs{x}} + \frac{2\pi e^{-1/\abs{y}}}{\abs{y}}.\]
	As a consequence, if $(x_j)$ and $(y_j)$ are sequences in $\widehat{\mathbb{R}}^2$ such that $x_j \to \infty$ and $y_j \to \infty$, then $d(x_j, y_j) \to 0$. 
	This is sufficient to conclude that $(\widehat{\mathbb{R}}^2, d)$ is homeomorphic to the Riemann sphere.
	
	In fact, by considering the pushforward of $\omega$ under the inversion map $x \mapsto x/\abs{x}^2$, we see that $(\widehat{\mathbb{R}}^2,d)$ is isometric to the metric space $(\widehat{\mathbb{R}}^2,\widetilde{d})$, where $\widetilde{d}$ is the metric generated by the conformal weight $\widetilde{\omega}(x) = e^{-\abs{x}}$. In particular, any ball in $(\widehat{\mathbb{R}}^2,d)$ centered at $\infty$ not containing the origin is bi-Lipschitz equivalent to a Euclidean disk.
	
	In \Cref{fig:geodesics}, a number of geodesics emanating from the point $p=(.3,0)$ are plotted. Observe that the length-minimizing path from $p$ to a point $q$ in the upper left region of the plot is the concatenation of the straight-line path from $p$ to the origin and the straight-line path from the origin to $q$. 
	
	This example illustrates how metric quasiconformality is not preserved in general under taking inverses or under precomposition with a quasisymmetry, as the following proposition shows. 
	
	\begin{proposition} \label{prop:example1_identity}
		Let $\iota\colon (\mathbb{R}^2,\abs{\cdot}) \rightarrow (\mathbb{R}^2,d)$ be the identity map, and let $h\colon \mathbb{R}^2 \rightarrow \mathbb{R}^2$ be the linear map defined by $h(x_1,x_2) = (x_1/2, x_2)$.
		\begin{enumerate}[label=(\alph*)]
		\item $\iota$ is MQC with $H=1$, as is its inverse. \label{item_a}
		\item $\iota$ is 1-QC. \label{item_b}
		\item $\iota$ is not I-QS. \label{item_c}
		\item $(\iota \circ h)^{-1}$ is MQC. \label{item_d}
		\item $\iota \circ h$ is not MQC. \label{item_e}
		\end{enumerate} 
	\end{proposition}
	\begin{proof}
	Claim \ref{item_a} is immediate for all $x \neq 0$ by virtue of $\omega$ being a conformal weight, and it also holds for $x=0$ by the radial symmetry of $\omega$.
	
	Claim \ref{item_b} is also immediate if we exclude $x=0$. However, observe that reciprocality condition \eqref{equ:reciprocality(3)} holds for the metric $d$ at the origin. Thus the geometric definition is unaffected by adding the origin back in, so the claim holds on all of $\mathbb{R}^2$.
	
	For claim \ref{item_c}, let $(t_j)$ be a sequence of positive numbers converging to zero, and let $y_j = (2t_j,0)$, $z_j = (t_j,0)$. Then $\abs{y_j - 0} = 2t_j$, $\abs{z_j - 0} = t_j$, $d(y_j,0) = \sqrt{e^{-1/t_j}}$, and $d(z_j,0) = e^{-1/t_j}$. But then $\abs{y_j - 0}/\abs{z_j-0} = 2$ while $d(y_j,0)/d(z_j,0) \rightarrow \infty$, violating the I-QS condition.
	
	For claim \ref{item_d}, note that $(\iota \circ h)^{-1}=h^{-1} \circ \iota^{-1}\colon (\mathbb{R}^2,d) \rightarrow (\mathbb{R}^2,\abs{\cdot})$ is the postcomposition of a MQC map by a QS map, which is always MQC.
	
	Claim \ref{item_e} follows from a variation of the argument for \ref{item_c}. Let $(t_j)$ again be a sequence of positive numbers converging to zero, and let $y_j = (t_j,0)$ and $z_j = (0,t_j)$. Then $h(y_j) = (t_j/2,0)$ and $h(z_j) = z_j$. This gives $d(h(y_j),0) = \sqrt{e^{-1/t_j}}$ and $d(z_j,0) = e^{-1/t_j}$, showing that $\iota \circ h$ is not MQC.
	\end{proof}
	
	Claim \ref{item_c} of \Cref{prop:example1_identity} can be strengthened to the following.
	\begin{proposition}
	There is no I-QS map $f\colon (\mathbb{R}^2,\abs{\cdot}) \rightarrow (\mathbb{R}^2,d)$.
	\end{proposition}
	\begin{proof}
	Suppose that such an I-QS map $f$ exists. Then $f^{-1}$ is also I-QS. Since metric quasiconformality is preserved under postcomposition by an I-QS map, it follows that $f^{-1}\circ \iota$ is an MQC map of the Euclidean plane. By the equivalence of definitions of quasiconformality in the Euclidean setting (for example, see \cite[Thm. 34.1]{Vai:71}), we conclude that $f^{-1}\circ \iota$ is QS and thus that $\iota$ itself is I-QS. This contradicts claim \ref{item_c} of \Cref{prop:example1_identity}.
	\end{proof}

	Note that the claims in \Cref{prop:example1_identity} all hold if we replace $\mathbb{R}^2$ with $\widehat{\mathbb{R}}^2$ equipped with the spherical metric. We also observe that 
	$(\mathbb{R}^2,d)$ is not upper 2-regular: The Hausdorff 2-measure of the ball $B_r=B(0,r)$, where $r \in [0,1]$, is given by
	\[\mathcal{H}^2(B_r) = \int_{B_r} \omega^2\,d\mathcal{L}^2 = 2\pi \int_0^R e^{-2/t}/t^3\,dt,\]
	where $R = -(\log r)^{-1}$. This evaluates to \[\mathcal{H}^2(B_r) = 2\pi e^{-2/R}\left(\frac{1}{4} + \frac{1}{2R}\right) = 2\pi r^2\left( \frac{1}{4} - \frac{\log r}{2}\right).\] 
	Since $-\log r \rightarrow \infty$ as $r \rightarrow 0$, we see that upper 2-regularity fails.
	
	
	\begin{figure}[t]
		\centering
		\includegraphics[width=0.6\textwidth]{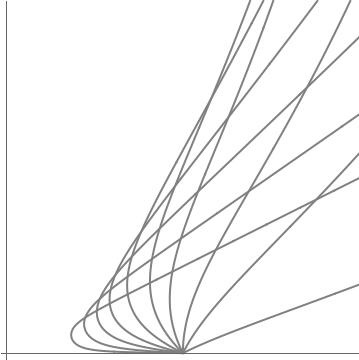}
		\caption{Geodesics emanating from the point $(.3,0)$}
		\label{fig:geodesics}
	\end{figure}

	\begin{proposition} \label{prop:llc_loewner}
	The space $(\widehat{\mathbb{R}}^2, d)$ is linearly locally connected. However, it is not a Loewner space.
	\end{proposition}
	The proof of linear local connectedness uses the following lemma. 
	\begin{lemma}\label{lemm:simply_conn}
		Let $x \in \mathbb{R}^2$ and $r>0$ be such that $B(x,r) \subset B(0, e^{-2})$. Then $B(x,r)$ is simply connected. 
	\end{lemma}
	\begin{proof}
		The claim is obvious when $x=0$, so we assume that $x \neq 0$. We argue by contradiction. Suppose that $B = B(x,r)$ is not simply connected. Since $(\widehat{\mathbb{R}}^2,d)$ is a geodesic space, all metric balls are connected. Hence the failure of simple connectivity implies that there exists a component $V$ of $\widehat{\mathbb{R}}^2 \setminus B$ not containing $\infty$. 
		
		Observe that $B(0,e^{-2})$ coincides with the Euclidean ball $B(0,1/2)$. In this region, $\omega$ is increasing as a function of the radius. Let $L$ be the Euclidean straight line which contains $x$ and the origin. The increasing property of $\omega$ implies that $L \cap B(0,e^{-2})$ is a geodesic segment. Thus $L \cap B(x,r)$ is connected, and in particular $V \cap L = \emptyset$. 
		
		It follows that $V$ is contained in one of the two open half-planes defined by the line $L$, denoted by $W$. Let $z \in V$ and let $S$ denote the Euclidean circle of radius $\abs{z}$ centered at the origin. Let $L'$ denote the Euclidean straight line containing $0$ and $z$. Then $W \setminus L'$ consists of two disjoint open sets $W_1, W_2$, where $x \in \partial W_1$. We observe that there exists a point $y \in S \cap B \cap \overline{W}_2$. A length-minimizing curve from $x$ to $y$ must cross $L'$ at some point $v$. However, the radial symmetry of $\omega$ implies that $d(v,z) \leq d(v,y)$, and thus that $d(x,z) \leq d(x,y)$. This gives a contradiction, and we conclude that $B$ is simply connected. 
	\end{proof}
	
	
	\begin{proof}[Proof of \Cref{prop:llc_loewner}]
	That $(\widehat{\mathbb{R}}^2, d)$ is linearly locally connected can be shown from \Cref{lemm:simply_conn} as follows. By Lemma 2.5 in \cite{BK:02}, it suffices to show that there exists $r_0>0$ and $\lambda \geq 1$ such that every ball $B(x,r)$ of radius $r \in (0,r_0)$ is contractible inside the ball $B(x,\lambda r)$. 
	
	Let $s =e^{-2}/4$ and let $L\geq 1$ be such that $(\widehat{\mathbb{R}}^2\setminus B(0,s), d)$ is $L$-bi-Lipschitz equivalent to a Euclidean disk. Let $r_0 = e^{-2}/(4L^2)$ and $\lambda = L^2$. For any $r \in (0,r_0)$ and $x \in \widehat{\mathbb{R}}^2$, the ball $B(x,\lambda r)$ is contained in $B(0,e^{-2})$ or it is contained in $(\widehat{\mathbb{R}}^2\setminus B(0,s), d)$. In the first case, $B(x,r)$ is simply connected by \Cref{lemm:simply_conn} and hence contractible. In the second case, the $L$-bi-Lispchitz equivalence of $(\widehat{\mathbb{R}}^2\setminus B(0,s), d)$ with a Euclidean disk implies that $B(x,r)$ is contractible inside $B(x,\lambda r)$. We conclude that $(\widehat{\mathbb{R}}^2, d)$ is linearly locally connected.  
		
	We now show that $(\widehat{\mathbb{R}}^2,d)$ is not Loewner. Let $E = (-\infty,0) \times\{0\}$ and let $F_t = [r_t,R_t] \times \{0\}$ for $t\in(0,1)$, where $r_t = -1/\log(t/2)$ and $R_t = -1/\log t$. Then $\dist(E,F_t) = \diam(F_t) = t$, so that $\Delta(E,F_t) = 1$ for all $t$. Observe that $\lim_{t \rightarrow 0} R_t/r_t = 1$. 
	
	Since the identity map $\iota \colon (\mathbb{R}^2,\abs{\cdot}) \to (\mathbb{R}^2,d)$ is 1-QC, the modulus of $\Gamma(E,F_t)$ relative to the metric $d$ is the same as the modulus of the same curve family relative to the Euclidean metric. These curve families arise classically in the Teichm\"uller ring problem \cite[Chapter III]{Ahl:66}. One can give an upper bound on their modulus as follows. Let $\Gamma_t$ denote the family of curves which span the open Euclidean annulus $A_t = A((r_t,0);R_t-r_t,r_t)$, where $t$ is sufficiently small so that $R_t < 2r_t$. For sufficiently small $t$, the annulus $A_t$ does not intersect $E$. The family $\Gamma_t$ majorizes $\Gamma(E,F_t)$ and has modulus $2\pi/\log(r_t/(R_t - r_t))$. 
	
	As $t \rightarrow 0$, we have that $\Mod \Gamma(E,F_t)$ goes to zero. Hence $(\mathbb{R}^2, d)$ is not Loewner.
	\end{proof}

The Loewner condition and linear local connectedness are conceptually similar in that they both rule out the existence of cusps and sequences of bottlenecks that become arbitrarily thin. In fact, the two properties are equivalent for the class of Ahlfors 2-regular metric spheres. This follows from Theorem 1.1 and Theorem 1.2 in \cite{BK:02} together with the quasisymmetric invariance of the Loewner condition \cite[Cor. 1.6]{Tys:98}. This example illustrates how, for metric spheres of finite Hausdorff 2-measure, linear local connectedness does not imply the Loewner condition without the assumption of Ahlfors regularity.

\subsection{An accumulation of spikes, I} \label{sec:spikes_1}
The purpose of this example is to give a metric surface $X$ so that the Hausdorff 2-measure on $X$ is upper 2-regular but $X$ fails to be I-LLC. Upper regularity implies, by \Cref{prop:reciprocality_criterion}, that there is a QC parametrization of $X$ by the Euclidean plane. However, $X$ does not admit an MQC parametrization by the Euclidean plane, as shown by the following simple lemma. 

\begin{lemma} \label{prop:mqc_llc}
Suppose there is an MQC map $g\colon \Omega \rightarrow X$, where $\Omega$ is a domain in $\mathbb{R}^2$. Then $X$ is I-LLC. 
\end{lemma}
\begin{proof}
Let $x \in X$ and $x' = g^{-1}(x)$. Let $R_x>0$ be sufficiently small so that $H_g(x',R) \leq 2H$ for all $R<R_x$. 

For small $r>0$, $g^{-1}(B(x,r)) \subset B(x',R_x)$. Let $y,z \in B(x,r)$, $y' = g^{-1}(y)$, $z' = g^{-1}(z)$, and $R' = \sup\{\abs{x'-w'}: w' \in g^{-1}(B(x,r))\}$. Then there is a curve $\mathcal{C}$ from $y'$ to $z'$ which is contained in $B(x',R')$. The metric quasiconformality implies that $g(\mathcal{C})$ is a curve from $y$ to $z$ contained in $B(x,2Hr)$. 

Similarly, let $y,z \in X \setminus B(x,r)$, with $y' = g^{-1}(y)$ and $z' = g^{-1}(z)$. Now, let $R' = \inf\{\abs{x'-w'}: w' \in \Omega \setminus g^{-1}(B(x,r))\}$. Connect $y'$ to $z'$ by a curve $\mathcal{C}$ in $\Omega \setminus B(x',R')$. Then metric quasiconformality implies that $g(\mathcal{C})$ is a curve from $y$ to $z$ contained in $X \setminus B(x,r/(2H))$. This establishes that $X$ is I-LLC. \end{proof}

We construct this example as a surface in $\mathbb{R}^3$ containing a sequence of spikes that become progressively smaller and converge to a point. For all $n \in \mathbb{N}$, let $t_n = 2^{-n}$, $h_n = 2^{-n/2}$, and $r_n = 2^{-2}\cdot 2^{-3n/2}$. The surface $X$ is constructed by removing each Euclidean disk $B((t_n,0),r_n)$ from $\mathbb{R}^2$, identified here with $\mathbb{R}^2 \times \{0\}$, and replacing it with a cone $S_n$ of height $h_n$. That is, $S_n$ has vertex $(t_n,0,h_n)$ and joins to $\mathbb{R}^2$ along the circle $S((t_n,0),r_n)$. We equip $X$ with the ambient Euclidean metric from $\mathbb{R}^3$, though the example works just as well if we were to take the induced length metric.

We check that $X$ is upper 2-regular. Let $x \in X$ and $r>0$. In the first case, assume that $r \leq \abs{x}/20$, where $\abs{\cdot}$ is the Euclidean norm in $\mathbb{R}^3$. A computation shows that $B(x,r)$ intersects at most one of the cones $S_n$. It is clear that $\mathcal{H}^2(B(x,r) \cap (\mathbb{R}^2 \times \{0\})) \leq \pi r^2$. By the elementary geometry of cones in $\mathbb{R}^3$, it also holds that $\mathcal{H}^2(B(x,r) \cap S_n) \leq \pi r^2$. We conclude that $\mathcal{H}^2(B(x,r) \leq 2\pi r^2$. 

In the second case, assume that $r> \abs{x}/20$. Then $B(x,r) \subset B(0,21r)$, writing $0$ to denote the origin in $\mathbb{R}^3$. For this, we compute 
\begin{align*} 
\mathcal{H}^2(B(0,2^{-n})) & \leq \pi 2^{-2n} + \sum_{k=n}^\infty \mathcal{H}^2(S_n) \\ 
& \leq \pi2^{-2n} + \pi \sum_{k=n}^\infty 2^{-3n/2}\sqrt{2^{-n} + 2^{-3n}} \\
& \leq \pi2^{-2n} + \pi \sum_{k=n}^\infty 2^{-3n/2}(2^{-n/2} + 2^{-3n/2}) \lesssim 2^{-2n}.
\end{align*} 
We deduce that $\mathcal{H}^2(B(0,21r)) \lesssim r^2$, and therefore that $X$ is upper 2-regular. 

Finally, the point $y_n= (t_n,0,h_n)$ lies outside the ball $B_n = B(0,\abs{y_n}/2)$. Any continuum connecting $y_n$ to the unbounded component of $\mathbb{R}^2 \setminus B_n$ must pass through the smaller ball $X \setminus B(0,2t_n)$. However, $\lim_{n \to \infty} t_n/\abs{y_n} = 0$, violating the I-LLC property. 

\subsection{An accumulation of spikes, II} \label{sec:spikes_2}
By modifying the previous example, we construct a space which is I-QS equivalent to the plane but not QS equivalent.   

We carry out the same construction as above, now taking $t_n = 2^{-n}$, $h_n = 2^{-n}$, and $r_n = 2^{-2} \cdot 2^{-2n}$. Instead of cones, we replace the Euclidean disks $B((t_n,0),r_n)$ 
with cylinders $C_n$ of height $h_n$. More precisely, $C_n= E_n \cup F_n$, where $E_n=\{(x_1,x_2,x_3): \, (x_1,x_2) \in S((t_n,0),r_n), \, 0\leq x_3 \leq h_n\}$ and $F_n=B((t_n,0),r_n)+(0,0,h_n)$. Again, we equip the resulting space $X$ with the restriction of the ambient Euclidean metric to $X$ to get $(X,d)$. 

The space $X$ is not LLC because the cylinders get progressively narrower; thus $X$ is not QS equivalent to the Euclidean plane. However, 
we claim that $X$ equipped with $\mu=\h^2$ satisfies the conditions of \Cref{thm:iqs} and therefore admits an I-QS map from $\R^2$. 

First, notice that for every $x \in X \setminus \{(0,0,0)\}$ there is $r_x>0$ so that $B(x,r_x) \subset X$ is 10-bi-Lipschitz equivalent to a planar disk. In particular, the conditions 
of \Cref{thm:iqs} hold for all such points $x$. 

We still need to verify the conditions of \Cref{thm:iqs} for $x=\mathbf{0}=(0,0,0)$. Take $r_{\mathbf{0}}=1/2$. The I-LLC condition follows from our choices of $t_n$, $h_n$ and $r_n$. Also, calculating as in \Cref{sec:spikes_1}, 
we conclude that $r^2 \lesssim \mathcal{H}^2(B(\mathbf{0},r)) \lesssim r^2$ for all $r>0$. Therefore, the $q$-metric on $X$ is comparable to the metric $d$, and $\mu$ is I-MM. 

Finally, we show that the I-Loewner condition is satisfied at $\mathbf{0}$. For a fixed $T>0$, let $s,t>0$ satisfy $s/t \leq T$. Let $n \in \mathbb{N}$ be such that $2^{-n-1} \leq s < 2^{-n}$. Consider two disjoint continua $E,F \subset X$ as in \Cref{def:iloew}. We make the observation that the cylinders $C_n$ and $C_{n-1}$ are separated by a distance of at least $2^{-n-1}$. Thus $F \cap(\mathbb{R}^2 \times \{0\}) \cap B(\mathbf{0},2^{-n+1})$ contains a continuum $F'$ of diameter at least $2^{-n-1}$. Next, we split into two cases. If $t \geq s$, then by similar reasoning $E \cap(\mathbb{R}^2 \times \{0\}) \cap B(\mathbf{0},2^{-n})$ contains a continuum $E'$ of diameter at least $2^{-n-2}$. If $t \leq s$, then take $E'$ to be a continuum in $E \cap(\mathbb{R}^2 \times \{0\}) \cap B(\mathbf{0},t)$ of diameter at least $t/16$. 

Then $d(E',F') \leq 2^{-n+2}$, and $E'$ and $F'$ have relative distance
\[\Delta(E',F') \leq \frac{2^{-n+2}}{\min\{2^{-n-2},t/16\}} \leq \max\left\{128T, 16\right\}. \]
Let $T' = \max\{128T,16\}$, so that $\Delta(E',F') \leq T'$. 
Consider the domain \[G = \left(\mathbb{R}^2 \setminus \bigcup_{n=1}^\infty \overline{B}((t_n,0),r_n)\right) \times \{0\} \subset X.\]
The domain $G$ is Loewner; let $\widetilde{\varphi}$ be the associated Loewner function. We have then the inequality
\[\Mod \Gamma(E,F) \geq \Mod \Gamma(E',F'; G) \geq \widetilde{\varphi}(T'). \]
We conclude that the I-Loewner condition is satisfied at $\bold{0}$. 


\subsection{Gluing a Grushin half-plane to a Euclidean half-plane}

The \emph{Grushin plane} is a basic example of a sub-Riemannian manifold. See \cite[Sec. 3.1]{Bel:96} for an overview. One approach to the Grushin plane, studied in \cite{Rom:16}, is given by the following definition. For each $\beta \in (0,1)$, the \emph{$\beta$-Grushin plane} is $\mathbb{R}^2$ equipped with the metric $\widetilde{d}$ obtained from the singular conformal weight $\widetilde{\omega}\colon \mathbb{R}^2 \to [0, \infty]$ defined by $\widetilde{\omega}(x) = \abs{x_1}^{-\beta}$. The standard Grushin plane is obtained by taking $\beta = 1/2$. 
Note that the standard Grushin plane does not have locally finite Hausdorff 2-measure. However, in the case when $\beta \in (0,1/2)$, it was shown in \cite{RV:17} and \cite{Wu:15} that the $\beta$-Grushin plane is bi-Lipschitz equivalent to the Euclidean plane. In particular, the $\beta$-Grushin plane is Ahlfors 2-regular. Moreover, the identity map $\mathbb{R}^2 \to (\mathbb{R}^2, \widetilde{d})$ is QS. A proof of this can be found in \cite[Thm. 4.3]{Rom:16}.

Here, we present a modified version of the Grushin plane. Let $\beta \in (0,1/2)$. Define the conformal weight $\omega\colon \mathbb{R}^2 \to [0,\infty]$ by
\[\omega(x) = \left\{ \begin{array}{ll} \abs{x_1}^{-\beta} & \text{ if } x_1 > 0 \\ 1 & \text{ if } x_1 \leq 0 \end{array} \right. .\]
Let $d$ denote the resulting metric. 

First, we establish a ball-box relationship. For all $r \leq 1$, let \[D_r = [-r,(1-\beta)r^{1/(1-\beta)}] \times [-r,r].\] Note that, for all $x_2 \in \mathbb{R}$, the straight-line curve from $(0,x_2)$ to $((1-\beta)r^{1/(1-\beta)},x_2)$ has length $r$. Observe further that $\omega \geq 1$ on $D_r$. From this, it follows that $d(x,0) \geq r$ for all $x \in \partial D_r$. Next, by considering the concatenation of the vertical line segment from 0 to $(0, x_2)$ with the horizontal line segment from $(0,x_2)$ to $x$, we see that $d(x,0) \leq 2r$ for all $x \in \partial D_r$. We conclude that 
\begin{equation}\label{eq:ball_box}
    B_d((0,0),r) \subset D_r \subset B_d((0,0),2r)
\end{equation}
for all $r \leq 1$. 

Next, observe that $\mathcal{H}^2(B_d(0,2r))$ is bounded from below by \begin{equation} \label{equ:grushin_area_of_ball} 
\int_{D_r} \omega^2 \,d\mathcal{L}^2 = 2r^2+ 2r\int_{0}^R t^{-2\beta}\,dt = 2r^2 + \frac{r^{(2-3\beta)/(1-\beta)}}{1-2\beta}.
\end{equation} 
For $\beta \in (0,1/2)$, the inequality $(2-3\beta)/(1-\beta)<2$ holds, from which we conclude that 
\[\liminf_{r \to 0} \frac{\mathcal{H}^2(B_d(x,r))}{r^2} = \infty\] 
for all $x$ lying on the vertical axis. On the other hand, \eqref{equ:grushin_area_of_ball} is an upper bound on $\mathcal{H}^2(B_d(x,r))$, showing that $(\mathbb{R}^2,d)$ has locally finite Hausdorff 2-measure.   

Since $\omega$ is constant on each vertical line, we see that metric balls are simply connected. In particular, $(\mathbb{R}^2,d)$ is LLC.

This example illustrates how a metric surface with locally finite 2-measure can violate infinitesimal upper 2-regularity at every point in a fairly large set, namely a nondegenerate continuum. 
Since any metric surface that is infinitesimally upper 2-regular is reciprocal, this suggests the following question.

\begin{question}
Is there a metric surface for which reciprocality condition \eqref{equ:reciprocality(3)} fails at every point on a nondegenerate continuum?
\end{question}

The space $(\mathbb{R}^2,d)$ in this example is reciprocal and hence does not answer this question. In fact, the identity map onto the Euclidean plane is 1-QC. This can be shown by a change of variables argument; see also Proposition 3.5 in \cite{GJR:17}, where the corresponding fact is proved for the $\beta$-Grushin plane. In contrast, we have the following.

\begin{proposition}
There is no MQC map from the Euclidean plane to $(\mathbb{R}^2,d)$.
\end{proposition}
\begin{proof}
Assume there is a MQC map $f\colon \mathbb{R}^2 \to (\mathbb{R}^2,d)$. Observe that the identity map $\iota\colon \mathbb{R}^2 \to (\mathbb{R}^2,d)$ is locally quasisymmetric outside of the vertical axis $Z$. This implies that $F = \iota^{-1} \circ f \colon \mathbb{R}^2 \to \mathbb{R}^2 $ is MQC outside of the set $f^{-1}(Z)$. By a classical removability theorem for planar quasiconformal mappings \cite[Thm. 35.1]{Vai:71}, it follows that $F$ is globally QS; see also Proposition 2.5 of \cite{GJR:17}. 

Let $x \in f^{-1}(Z)$. By quasisymmetry, there exists $H \geq 1$ such that \[B_{\text{Euc}}(F(x),s(r)) \subset F(B_{\text{Euc}}(x,r)) \subset B_{\text{Euc}}(F(x),Hs(r))\] for all $r>0$, where $s(r) = \inf\{|F(x) - F(y)|: y \in \mathbb{R}^2 \setminus B_{\text{Euc}}(x,r)\}$. Comparing this with the ball-box relationship \eqref{eq:ball_box}, we conclude that $f$ is not MQC. This is a contradiction.
\end{proof}
A similar argument shows that there is no MQC map from $(\mathbb{R}^2,d)$ to $\mathbb{R}^2$. 

\subsection*{Acknowledgment} We thank the referee for a careful reading of the paper and detailed feedback, which helped us improve the text significantly. 

\bibliographystyle{abbrv}
\bibliography{imm_final} 

\vspace{1em}
\noindent
Department of Mathematics and Statistics, University of Jyvaskyla, P.O.
Box 35 (MaD), FI-40014, University of Jyvaskyla, Finland.\\

\emph{E-mail:} \settowidth{\hangindent}{\emph{aaaaaaaaa}} Kai Rajala: \textbf{kai.i.rajala@jyu.fi} \\ Martti Rasimus: \textbf{martti.i.rasimus@jyu.fi} \\ Matthew Romney: \textbf{matthew.d.romney@jyu.fi}

\end{document}